\documentclass[final]{siamltex}
\usepackage{amsmath}
\usepackage{amssymb}
\usepackage{amsfonts}

\usepackage{ushort}

\newcommand{\bunderline}[1]{\underline{#1\mkern-4mu}\mkern4mu }

\def\B{\mathcal B}

\def\F{\mathcal F}

\def\ereals{\overline{\mathbb{R}}}
\def\downto{{\raise 1pt \hbox{$\scriptstyle \,\searrow\,$}}}

\def\u{\mathnormal u}
\def\x{\mathnormal x}

\def\R{\mathbb R}

\def\reals{\mathbb R}
\def\epi{\mathop{epi}}

\newcommand{\gph}{\mathop {\rm gph}}
\def\epi{\mathop{\rm epi}\nolimits}
\def\cl{\mathop{\rm cl}\nolimits}

\def\dom{\mathop{\rm dom}\nolimits}

\def\mli{\mathop{\rm \mu\text{-}liminf}}

\def\interior{\mathop{\rm int}}
\def\rinterior{\mathop{\rm rint}}

\def\aff{\mathop{\rm aff}}

\def\uball{\mathbb{B}}

\def\minimize{\mathop{\rm minimize}\limits}

\begin{document}
\title{Duality in convex problems of {B}olza over functions of bounded variation}
\author{Teemu Pennanen\thanks{Department of Mathematics, King's College London, Strand, London WC2R 2LS, United Kingdom ({\tt teemu.pennanen@kcl.ac.uk}).}
	\and Ari-Pekka Perkki\"o\thanks{Department of Mathematics and Systems Analysis, Aalto University, P.O. Box 11100, FI-00076 Aalto, Finland ({\tt ari-pekka.perkkio@aalto.fi}). Corresponding author.} 
}

\maketitle

\begin{abstract}
This paper studies convex problems of Bolza in the conjugate duality framework of Rockafellar. We parameterize the problem by a general Borel measure which has direct economic interpretation in problems of financial economics. We derive a dual representation for the optimal value function in terms of continuous dual arcs and we give conditions for the existence of solutions. Combined with well-known results on problems of Bolza over absolutely continuous arcs, we obtain optimality conditions in terms of extended Hamiltonian conditions.
\end{abstract}


\pagestyle{myheadings}
\thispagestyle{plain}
\markboth{T. PENNANEN AND A.-P. PERKKI\"O}{DUALITY IN CONVEX PROBLEMS OF BOLZA}

\section{Introduction} 

Problems of Bolza were introduced a century ago as a general class of problems in the calculus of variations \cite{bol13}. In a sequence of papers in the 70's, Rockafellar extended the theory to possibly nonsmooth and extended real-valued convex Lagrangian's and end-point penalties. This extension allows for treating convex problems of optimal control under the same framework. Rockafellar's original formulation was over absolutely continuous arcs \cite{roc70b}, but soon after, he generalized it to arcs of bounded variation~\cite{roc72a,roc76b}. We refer the reader to \cite[Section~6.5]{mor6b} for a general account of the history of optimal control and the calculus of variations. 

The present paper extends the theory of convex problems of Bolza in two directions. First, we relax the continuity assumptions on the domain of the Hamiltonian using recent results of Perkki\"o~\cite{per12} on conjugates of convex integral functionals. Second, we parameterize the primal problem with a general Borel measure that shifts the derivative rather than the state. Our parameterization is of interest in financial economics where the parameter may represent e.g.\ endowments and/or liabilities of an economic agent. The relaxed continuity requirements allow discontinuous state constraints both in the primal and the dual.

Given $T>0$, let $X$ be the space of left-continuous functions $x:\reals_+\to\reals^d$ of bounded variation such that $x$ is constant after $T$. The space $X$ may be identified with $\reals^d\times M$ where $M$ is the space of $\reals^d$-valued Radon measures on $[0,T]$. Indeed, given $x\in X$ there is a unique $\reals^d$-valued Radon measure $Dx$ on $[0,T]$ such that $x_t=x_0+Dx([0,t))$ for all $t\in[0,T]$ and $x_t=x_0+Dx([0,T])$ for $t>T$; see e.g.~\cite[Theorem~3.29]{fol99}. The value of $x\in X$ on $(T,\infty)$ will be denoted by $x_{T+}$. 

Given an atomless strictly positive\footnote{A measure $\mu$ is strictly positive if $\mu(O)>0$ for every nonempty open $O$.} Radon measure $\mu$ on $[0,T]$, a proper convex normal integrand $K:\R^d\times\R^d\times[0,T]\to\ereals$ and a proper convex lower semicontinuous function $k:\R^d\times\R^d\to\ereals$, we will study the parametric optimization problem
\begin{equation}\label{P_u}\tag{P$_u$}
\minimize\quad J_K(x,Dx + u)+k(x_{0},x_{T+})\quad\text{over $x\in X$},
\end{equation}
where $u\in M$ and $J_K:X\times M\to\ereals$ is given by
\[
J_K(x,\theta)=\int   K_t(x_t,(d\theta^a/d\mu)_t)d\mu_t+\int   K^\infty_t(0,(d\theta^s/d|\theta^s|)_t)d|\theta^s|_t.
\]
Here $\theta^a$ and $\theta^s$ denote the absolutely continuous and the singular parts of $\theta$ with respect to $\mu$, $|\theta^s|$ denotes the total variation of $\theta^s$ and $K^\infty_t$ the recession function of $K_t$; see the appendix. Throughout this paper, we define the integral of a measurable function as $+\infty$ unless its positive part is integrable. Similarly, the sum of finite collection of extended real numbers is defined as $+\infty$ if any of the terms equals $+\infty$. It follows that $J_K$ as well as the objective in \eqref{P_u} are well-defined extended real-valued functions on $X\times M$.

When $u=0$ and the minimization is restricted to the space $AC$ of absolutely continuous functions with respect to $\mu$, problem \eqref{P_u} can be written in the more familiar form
\begin{equation}\label{P_AC}\tag{P$_{AC}$}
\minimize \quad \int K_t(x_t,\dot x_t)d\mu_t+k(x_{0},x_{T})\quad\text{over $x\in AC$},
\end{equation}
where $\dot x$ denotes the Radon--Nikodym derivative of $Dx$ with respect to $\mu$. Such problems have been extensively studied since \cite{roc70b} (often in the case where $\mu$ is the Lebesque measure). Allowing $K$ and $k$ to be extended real-valued, various more traditional problems in calculus of variations and optimal control can be written in the above form; see \cite{roc70b,roc78} for details. Problems of the form \eqref{P_u} with $u=0$ extend \eqref{P_AC} by allowing for discontinuous trajectories. In the context of optimal control, discontinuous trajectories correspond to impulsive control. Rockafellar~\cite{roc76b} developed a duality theory for problems of the form \eqref{P_u} with $u=0$ in the case where\footnote{Throughout this paper, $\delta_C$ denotes the {\em indicator function} of a set $C$, i.e.\ $\delta_C(x)=0$ if $x\in C$ and $\delta_C(x)= +\infty$ otherwise.} $k=\delta_{\{(a,b)\}}$.

Much as in \cite{roc70b,roc71b,roc72a,roc76b,roc78}, we will study \eqref{P_u} by embedding it in the general conjugate duality framework of \cite{roc74}. We give sufficient conditions under which the infimum in \eqref{P_u} is attained for every $u$ and the {\em value function}
\begin{equation*}
  \varphi(u)=\inf_{x\in X}\{J_K(x,Dx+u)+k(x_{0},x_{T+})\}
\end{equation*}
of \eqref{P_u} has the dual representation
\[
\varphi(u)=\sup_{y\in C\cap X}\{\langle u,y\rangle - J_{\tilde K}(y,Dy) - \tilde k(y_0,y_T)\},
\]
where $C$ denotes the space of continuous function on $[0,T]$ and $\tilde K$ and $\tilde k$ are given in terms of the conjugates of $K_t$ and $k$ as 
\begin{align*}
  {\tilde K}_t(y,v)=K^*_t(v,y)&=\sup_{\x,\u\in\R^d}\{x\cdot v + u\cdot y - K_t(x,u) \},\\
  \tilde k(\tilde a,\tilde b)=k^*(\tilde a,-\tilde b)&=\sup_{a,b\in\R^d}\{a\cdot \tilde a - b\cdot \tilde b -k(a,b)\}.
\end{align*}
This paper relaxes the continuity assumptions made in \cite{roc72a,roc76b} on the domain of the associated Hamiltonian
\[
H_t(x,y)=\inf_{u\in\R^d}\{K_t(x,u) - u\cdot y\}.
\]
This turns out to have significant consequences in certain problems of financial economics where the continuity relates to the behavior of financial markets; see \cite{pp13b} for details. We also show that that our relaxed continuity assumptions allow for optimality conditions in terms of an extended Hamiltonian equation. Combined with the results of \cite{roc71b} on problems of Bolza over absolutely continuous arcs, we obtain necessary and sufficient conditions of optimality in \eqref{P_u} with $u=0$.

\section{Conjugate duality}\label{sec:cd}

A set-valued mapping $S:[0,T]\rightarrow\R^d$ is {\em measurable} if the preimage $S^{-1}(A):=\{t\in [0,T]\mid S_t\cap A\neq\emptyset\}$ of every open $A\subset \R^d$ is measurable. An extended real-valued function $h$ on $\reals^n\times [0,T]$ is a {\em proper convex normal integrand} if the set-valued mapping $t\mapsto\epi h_t(\cdot)$ is closed convex-valued and measurable, and $h_t(\cdot)$ is proper for all $t$. By \cite[Corollary~14.34]{rw98}, this implies that $h$ is an $\B(\reals^d)\otimes \B([0,T])$-measurable function on $\reals^d\times [0,T]$, so $t\mapsto h_t(x_t)$ is an $\F$-measurable extended real-valued function and 
\[
I_h(x)=\int  h_t(x_t)d\mu_t
\]
is well-defined for every $\B([0,T])$-measurable $x:[0,T]\to\reals^d$. For every $t$, the recession function $x\mapsto h_t^\infty(x)$ is a closed and sublinear convex function; see the appendix. By \cite[Exercise 14.54]{rw98}, $h^\infty$ is a convex normal integrand.

We will study \eqref{P_u} in the conjugate duality framework of Rockafellar~\cite{roc74}. To this end, write it as
\[
\minimize\quad f(x,u)\quad \text{over}\quad x\in X,
\]
where
\begin{equation*}
f(x,u)=J_K(x,Dx+u)+k(x_0,x_{T+}).
\end{equation*}
Since $K$ is a convex normal integrand, we see that $f$ is well-defined on $X\times M$. The convexity of $K$ and $k$ implies the convexity of $f$ on $X\times M$, which in turn implies that the optimal value function
\begin{equation*}
\varphi(u)=\inf_{x\in X} f(x,u)
\end{equation*}
is convex on $M$; see e.g.\ \cite[Theorem~1]{roc74}.

The bilinear form
\[
\langle u,y\rangle := \int  y_t du_t
\]
puts $M$ in separating duality with the space $C$ of $\R^d$-valued continuous functions on $[0,T]$. Indeed, if we equip $C$ with the supremum norm, Riesz representation theorem says that $M$ may be identified with the Banach dual of $C$ through the representation $y\mapsto\langle u,y\rangle$; see e.g.\ \cite[Theorem~7.17]{fol99}. Similarly the bilinear form
\[
\langle x,v\rangle:=x_0\cdot v_{-1}+\int v_tdx_t
\]
puts $X$ in separating duality with the space $V:=\reals^d\times C$ of continuous functions on $\{-1\}\cup[0,T]$. The weak topology on $X$ will be denoted by $\sigma(X,V)$. We will make repeated use of the integration by parts formula
\[
\int v_tdx_t=x_{T+}\cdot v_T-x_0\cdot v_0 - \int x_tdv_t,
\]
which is valid for any $x\in X$ and any $v\in C$ of bounded variation; this can be deduced, e.g., from \cite[Theorem VI.90]{dm82} or Folland~\cite[Theorem 3.36]{fol99}.

The {\em Lagrangian} associated with \eqref{P_u} is the convex-concave function on $X\times Y$ defined by
\begin{equation*}
L(x,y)=\inf_{u\in M}\{f(x,u)-\langle u,y\rangle \}.
\end{equation*}
The conjugate of $\varphi$ can be expressed as
\begin{align*}
\varphi^*(y) &= \sup_{u\in M}\{\langle u,y\rangle -\varphi(u)\}\\
&= \sup_{u\in M, x\in X}\{\langle u,y\rangle -f(x,u)\}\\
&= -g(y),
\end{align*}
where
\[
g(y) := \inf_{x\in X} L(x,y).
\]
If $\varphi$ is closed (i.e.\ either proper and lower semicontinuous or a constant function), the biconjugate theorem (see e.g.\ \cite[Theorem~5]{roc74}) gives the dual representation
\begin{align*}
\varphi(u) = \sup_{y\in C}\{\langle u,y\rangle + g(y)\}.
\end{align*}
Clearly, $g(y)=-f^*(0,y)$, where $f^*$ is the conjugate of $f$. We always have
\[
f^*(v,y) =\sup_{x\in X}\{\langle x,v\rangle-L(x,y)\},
\]
and, as soon as $f$ is closed in $u$, 
\[
f(x,u) =\sup_{y\in Y}\{\langle u,y\rangle + L(x,y)\}.
\]

Our first goal is to derive a more concrete expression for $L$.  This will involve  the \emph{Hamiltonian} $H:\R^d\times\R^d\times[0,T]\to\ereals$ defined by
\begin{equation*}
H_t(x,y)=\inf_{u\in\R^d}\{K_t(x,u) - u\cdot y\}.
\end{equation*}
The Hamiltonian is convex in $x$ and concave in $y$. The function $t\mapsto H_t(x_t,y_t)$ is measurable for every $x\in X$ and $y\in C$. Indeed, by \cite[Proposition 14.45 and Theorem 14.50]{rw98}, $(y,t)\mapsto -H_t(x_t,y)$ is a normal integrand for every $x\in X$, so the measurability follows from that of $y$. The integral functional
\[
I_H(x,y)=\int   H_t(x_t,y_t)d\mu
\]
is thus well defined on $X\times C$. Again, we set $I_H(x,y)=+\infty$ unless the positive part of the integrand is integrable. The function $I_H$ is convex in $x$ and concave in $y$. The set $\dom H_t:=\dom_1 H_t\times\dom_2 H_t$ where
\begin{align*}
  \dom_1 H_t &=\{x\in\R^d\,|\, H_t(x,y)<+\infty\ \forall y\in\R^d \},\\
  \dom_2 H_t &=\{y\in\R^d\,|\, H_t(x,y)>-\infty\ \forall x\in\R^d\}
\end{align*}
is known as the {\em domain} of $H_t$. This set is nonempty for all $t$, because $K_t(\cdot,\cdot)$ is proper \cite[Theorem 34.2]{roc70a}. The domain of $I_H$ is defined similarly. 

Recall that a set-valued mapping $S$ from $[0,T]$ to $\reals^d$ is {\em inner semicontinuous} (isc) if the preimage of every open set is open; see \cite[Chapter~5]{rw98}. Following \cite{per12}, we define
\begin{align*}
(\mli S)_t=\{y\in\R^d\mid \forall\, A\in\mathcal H_y^o\ \exists O\in\mathcal H_t:\ \mu(S^{-1}(A)\cap O)=\mu(O)\},
\end{align*}
where $\mathcal H_y^o$ is the collection of open neighborhoods of $y\in\R^d$ and $\mathcal H_t$ is the collection of all neighborhoods of $t\in[0,T]$. A mapping $S$ is {\em outer $\mu$-regular} if $(\mli S)_t\subseteq\cl S_t$. If $S$ is outer $\mu$-regular, then we have that $y_t\in\cl S_t$ for all $t$ whenever $y\in C$ is such that $y_t\in \cl S_t$ $\mu$-almost everywhere; see \cite[Theorem 1]{per12}. By \cite[Theorem 2]{per12}, the converse implication holds when $S$ is isc convex-valued with $\interior S_t\neq\emptyset$ for all $t$. Outer $\mu$-regularity together with inner semicontinuity generalize the {\em full lower semicontinuity} condition used in \cite{roc72a,roc76b}. We denote the relative interior of a set $A$ by $\rinterior A$.

\begin{theorem}\label{thm:L}
Assume that 
\begin{enumerate}
\item [1.] $t\mapsto\dom_2 H_t$ is isc and outer $\mu$-regular,
\item [2.] $\{y\in C\mid y_t\in\rinterior\dom_2 H_t\,\forall t\}\subset \dom_2 I_H$,
\item [3.] For every $x\in \dom_1 I_H$ there exist $\alpha\in L^1$ and $w\in L^1$ with
\[
H_t(x_t,y)\le -y\cdot w_t+\alpha_t\ \mu\text{-a.e.,}
\]
\end{enumerate}
Then $f$ is closed in $u$, and the Lagrangian can be expressed as
\[
L(x,y) = \begin{cases}
I_H(x,y)+\langle Dx,y\rangle+k(x_0,x_{T+})\quad&\text{if } x\in\dom_1 I_H,\\
+\infty\quad&\text{otherwise}.
\end{cases}
\]
In particular $f$ is proper whenever $\dom_1 I_H\neq\emptyset$. Moreover, if
\begin{enumerate}
\item [4.]For every $y\in\dom_2 I_H$ there exist $\beta\in L^1$ and $z\in L^1$ with
\[
H_t(x,y_t)\ge x\cdot z_t-\beta_t\ \mu\text{-a.e.},
\] 
\end{enumerate}
then $f:X\times U\rightarrow\ereals$ is closed.
\end{theorem}

\begin{proof}
By definition, 
\begin{align*}
L(x,y) &= \inf_{u\in M}\{J_K(x,Dx+u)+k(x_0,x_{T+})-\langle u,y\rangle\}\\
&= \inf_{u\in M}\{J_K(x,u)-\langle u,y\rangle\} + \langle Dx,y\rangle + k(x_0,x_{T+}).
\end{align*}
Assume first that $x\notin\dom_1 I_H$ so that there is a $\tilde y\in C$ such that $I_H(x,\tilde y)=+\infty$. Since
\[
K_t(x,u) \ge H_t(x,\tilde y_t) + u\cdot \tilde y_t\quad \forall x,u\in\reals^d,
\]
we have $J_K(x,Dx+u)=+\infty$ for all $u\in M$, so $L(x,y)=+\infty$ and the given expression for the Lagrangian is valid.

Assume now that $x\in\dom_1 I_H$. We may redefine $x_0$ and $x_{T+}$ so that $k(x_0,x_{T+})<+\infty$. To justify the expression for $L$ and that $f$ is proper and closed in $u$, it suffices to show that the functions $J_K(x,\cdot)$ and $-I_H(x,\cdot)$ are proper and conjugate to each other. By condition 3, there is a Borel $\mu$-null set $N$ with $\{t\mid x_t\notin\dom_1 H_t\}\subseteq N$. Since $-H_t(x_t,\cdot)$ is a normal integrand and $t\mapsto\dom_2 H_t$ is measurable (see \cite[Exercise 14.9]{rw98}), it follows that
\[
h_t(y)=
\begin{cases}
  -H_t(x_t,y) & \text{if $t\notin N$},\\
  \delta_{\cl\dom_2 H_t}(y) &\text{if $t\in N$},
\end{cases}
\]
is a normal integrand and $I_h=-I_H(x,\cdot)$. Clearly, $h^*_t(u)=K_t(x_t,u)$ for all $t\notin N$. Since, by \cite[Theorem~34.2]{roc70a}, $\dom_2 H_t=\{y\mid \exists\ v:\ (v,y)\in\dom K_t^*\}$ we have, by \cite[Theorem 13.3]{roc70a}, that $(h^*_t)^\infty(u)=K^\infty_t(0,u)$ for all $t$ and thus,
\[
J_K(x,\theta) = \int h^*_t((d\theta/d\mu)_t)d\mu_t + \int (h^*_t)^\infty((d\theta^s/d|\theta^s|)_t)d|\theta^s|_t.
\]

By \cite[Theorem 34.3]{roc70a}, $\cl\dom h_t=\cl\dom_2 H_t$ for all $t$, so $t\mapsto\dom h_t$ is isc and outer $\mu$-regular. The mapping $t\mapsto\rinterior\dom h_t$ is also isc and convex-valued, so, by \cite[Theorem 3.1''']{mic56}, there is a $\bar y\in C$ with $\bar y_t\in\rinterior\dom h_t$ for all $t$. Thus, condition 2 implies that $\dom I_h\neq\emptyset$. By condition 3, $K_t(x_t,w_t)\le \alpha_t$ $\mu$-a.e., so, by choosing $d\theta/d\mu=w$ and $\theta^s=0$, we see that $J_K(x,\theta)<\infty$. Hence all the assumptions of \cite[Theorem 3]{per12} are met, so $J_K(x,\cdot)$ is a proper closed convex function with the conjugate $-I_H(x,\cdot)$. 
 
Assume now that condition 4 holds. It remains to show that $f$ is lower semicontinuous on $X\times M$. By the above,
\[
f(x,u)= \sup_{y\in C}\{\langle Dx+u,y\rangle +I_H(x,y)\} +k(x_0,x_{T+}).
\]
We start by showing that the supremum can be restricted to $y\in C$ with $y_t\in\rinterior\dom_2 H_t$ for all $t$. If $x_t$ does not belong to $\dom_1 H_t$ almost everywhere, then, by \cite[Theorem~34.3]{roc70a}, $H_t(x_t,\bar y_t)=+\infty$ on a set of positive measure, so $I_H(x,\bar y)=+\infty$. On the other hand, if $x_t\in\dom_1 H_t$ $\mu$-a.e.\ and if $I_H(x,y)>-\infty$, then, by \cite[Theorem 34.3]{roc70a}, $y_t\in\cl\dom_2 H_t$ $\mu$-a.e., so outer $\mu$-regularity implies that $y_t\in\cl\dom_2 H_t$ for all $t$. Defining $y^\nu=\frac{1}{\nu}\bar y+(1-\frac{1}{\nu})y$, we have $y^\nu_t\in\rinterior\dom_2 H_t$ for all $\nu$ and,  by concavity, $I_H(x,y^\nu)\ge \frac{1}{\nu}I_H(x,\bar y)+(1-\frac{1}{\nu})I_H(x,y)$.

When $y\in C$ with $y_t\in\rinterior\dom_2 H_t$ for all $t$, the function $H_t(\cdot,y_t)$ is lsc for all $t$ by \cite[Theorem 34.2]{roc70a}, so, by \cite[Proposition 14.47]{rw98}, $(x,t)\mapsto H_t(x,y_t)$ is a normal integrand. By \cite[Theorem~3C]{roc76}, condition 4 implies that $I_H(\cdot,y)$ is lsc on $(L^\infty,\sigma(L^\infty,L^1))$. 

To finish the proof, it suffices to show that the embedding $(X,\sigma(X,V))\hookrightarrow (L^\infty,\sigma(L^\infty,L^1))$ is continuous. Let $\bar w\in L^1$, $\varepsilon>0$ and $\bar z_t=\int_{[0,t]}\bar w_td\mu_t$. Integration by parts gives
\begin{align*}
\{x\in X\mid |\int x_t\cdot \bar w_td\mu_t|<\varepsilon\}
&=\{x\in X\mid |\int x_t d\bar z_t|<\varepsilon\}\\
&=\{x\in X\mid |\bar z_T\cdot x_{T+}-\int \bar z_tdx_t|<\varepsilon\},\\
&=\{x\in X\mid |\bar z_T\cdot x_0+\int (\bar z_T-\bar z_t)dx_t|<\varepsilon\}.
\end{align*}
Thus, since $\bar z$ is continuous, $\sigma(L^\infty,L^1)$-open sets are $\sigma(X,V)$-open.
\qquad\end{proof}

Conditions 1 and 2 in Theorem~\ref{thm:L} are needed to apply the results of \cite{per12} on convex conjugates of integral functionals. If $t\mapsto\dom_2 H_t$ is isc with $\interior\dom_2 H_t\neq\emptyset$ for all $t$, then, under conditions 2 and 3, outer $\mu$-regularity of $t\mapsto\dom_2 H_t$ is necessary for the conclusion of the theorem to hold. This follows by applying \cite[Theorem 3]{per12} to $I_h$ in the proof above.

We will next derive a more explicit expression for the conjugate of $f$. By \cite[Theorem 14.50]{rw98}, the function
\[
{\tilde K}_t(y,v)=K^*_t(v,y)=\sup_{\x,\u\in\R^d}\{x\cdot v + u\cdot y - K_t(x,u) \}.
\]
is a proper convex normal integrand. The functional 
\[
J_{\tilde K}(y,\theta)=\int   {\tilde K}_t(y_t,(d\theta^a/d\mu)_t)d\mu+\int  {\tilde K}^\infty_t(0, (d\theta^s/d|\theta^s|)_t)d|\theta^s|
\]
is thus well defined on $C\times M$. We also define
\[
\tilde k(\tilde a,\tilde b)=k^*(\tilde a,-\tilde b)=\sup_{a,b\in\R^d}\{a\cdot \tilde a - b\cdot \tilde b -k(a,b)\}.
\]

A function $x:\reals\to\reals^d$ is left-continuous if and only if it is continuous with respect to the topology $\tau_l$ generated by sets of the form $\{(s,t]\mid s<t\}$. We will say that a set-valued mapping $S$ is {\em left-inner semicontinuous} (or left-isc) if it is isc with respect to $\tau_l$. Similarly, $S$ is said to be {\em left-outer $\mu$-regular} if it is outer $\mu$-regular with respect to $\tau_l$. By \cite[Theorem 2]{per12}, a left-isc convex-valued mapping $S$ with $\interior\dom S_t\neq\emptyset$ for all $t$ is left-outer $\mu$-regular if and only if $x_t\in\cl S_t$ for all $t$ whenever $x$ is a left-continuous function with $x_t\in \cl S_t$ $\mu$-almost everywhere. We denote by $\mathbb B(x,r)$ the open ball with center $x\in\R^d$ and radius $r>0$.

\begin{theorem}\label{thm:pdo}
In addition to hypotheses of Theorem~\ref{thm:L}, assume that
\begin{enumerate}
\item $t\mapsto\dom_1 H_t$ is left-isc and left-outer $\mu$-regular,
\item $\emptyset\neq\{x\in X\mid \exists r>0:\, \mathbb B(x_t,r)\subset\dom_1 H_t\ \forall t\}\subset\dom_1 I_H$.
\end{enumerate}
Then
\[
f^*(v,y)=  
\begin{cases}
J_{\tilde K}(y,D\tilde v) + \tilde k(v_{-1}+\tilde v_0,\tilde v_T)\quad&\text{if $\tilde v \in C\cap X$},\\
+\infty\quad&\text{otherwise},
\end{cases}
\]
where $\tilde v_t=y_t-v_t$ for $t\in[0,T]$.
\end{theorem}

\begin{proof}
By Theorem~\ref{thm:L},
\begin{align*}
f^*(v,y) &= \sup_{x \in X}\{\langle x,v\rangle - L(x,y)\}\\
&= \sup_{x \in X}\{\langle x,v\rangle - I_H(x,y)-\langle Dx,y\rangle-k(x_0,x_{T+})\}\\
&= \sup_{x \in X}\{x_0\cdot v_{-1}-I_H(x,y)-\langle Dx,\tilde v\rangle-k(x_0,x_{T+})\}.
\end{align*}

Assume first that $y\notin\dom_2 I_H$ so that there is an $\tilde x\in X$ such that $I_H(\tilde x,y)=-\infty$. Since $I_H(\tilde x,y)$ is independent of the end points of $\tilde x$, we get $f^*(v,y)=+\infty$. The expression for $f^*$ then clearly holds if $\tilde v \notin C\cap X$. Since $\tilde K_t(y,v)\ge\tilde x_t\cdot v-H_t(\tilde x_t,y)$, we have $J_{\tilde K}(y,D\tilde v)=+\infty$, so the expression is valid also for $\tilde v\in C\cap X$. We may thus assume that $y\in\dom_2 I_H$.

Let $\bar x$ belong to the set in condition 2. Redefining $\bar x_0$ and $\bar x_{T+}$, we may assume that $k(\bar x_0,\bar x_{T+})<\infty$. Since $y\in\dom_2 I_H$, we have that $I_H(\cdot,y)$ is proper on $X$. In view of condition 2, \cite[Theorem 2]{roc71} implies that there is an $r>0$ and an $\alpha\in\R$ such that $I_H(\bar x+x,y)\leq\alpha$ whenever $x\in X$ with $x_t\in\uball(0,r)$ for all $t$. Therefore,

\begin{align*}
f^*(v,y)&\ge\sup_{x\in C^1_c}\{\bar x_0\cdot v_{-1}-I_H(\bar x+x,y)- \langle D(\bar x+x),\tilde v\rangle- k(\bar x_0,\bar x_{T+}) \,|\, x_t\in\uball(0,r)\, \forall t\}\\
&\ge \bar x_0\cdot v_{-1} - \alpha - \langle D\bar x,\tilde v\rangle  - k(\bar x_0,\bar x_{T+}) + \sup_{x\in C^1_c}\{-\langle Dx,\tilde v\rangle\mid x_t\in\uball(0,r)\ \forall t\},
\end{align*}
where $C^1_c$ is the set of continuously differentiable $\R^d$-valued functions with compact support in $(0,T)$. By \cite[Proposition 3.6]{afp00}, the last supremum equals the total variation of $\tilde v$ on $(0,T)$ and consequently  $f^{*}(v,y)=+\infty$ unless $\tilde v$ is of bounded variation on $[0,T]$. When $\tilde v$ is of bounded variation, integration by parts gives
\begin{align*}
  f^*(v,y) &= \sup_{x\in X}\{x_0\cdot v_{-1} - I_H(x,y)-\int \tilde v_tdx_t-k(x_0,x_{T+})\}\\
&= \sup_{x\in X}\{\int x_td\tilde v_t - I_H(x,y) + x_0\cdot (v_{-1}+\tilde v_0) - x_{T+}\cdot \tilde v_T - k(x_0,x_{T+})\}\\
&= \sup_{x\in X}\{\int x_td\tilde v_t-I_H(x,y)\} + \sup_{x\in X}\{x_0\cdot (v_{-1}+\tilde v_0) - x_{T+}\cdot \tilde v_T - k(x_0,x_{T+})\}\\
&= \sup_{x\in X}\{\int x_td\tilde v_t-I_H(x,y)\} + \tilde k(v_{-1}+\tilde v_0,\tilde v_T).
\end{align*}
Analogously to the proof of Theorem~\ref{thm:L}, we can restrict the supremum to the set $\{x\in X\mid x_t\in\interior\dom_1 H_t\ \forall t\}$. Thus, by \cite[Corollary 34.2.1]{roc70a},
\[
f^*(v,y) = \sup_{x\in X}\{\int x_td\tilde v_t-I_{\bunderline H}(x,y)\} + \tilde k(v_{-1}+\tilde v_0,\tilde v_T),
\]
where $\bunderline H_t$ denotes the closure of $H_t$ with respect to $x$. The rest of the proof is analogous to the proof of Theorem~\ref{thm:L} except that instead of \cite[Theorem~3]{per12} we use \cite[Theorem~4]{per12} on integral functionals of left-continuous functions of bounded variation.
\qquad\end{proof}

\section{A closedness criterion}\label{sec:cc}

This section gives sufficient conditions for the closedness of $\varphi$ by applying general results on the conjugate duality framework derived in the appendix. To this end, we write $\varphi$ as
\[
\varphi(u)=\inf_{a\in\R^d}\varphi_0(a,u),
\]
where
\[
\varphi_0(a,u)=\inf_{x\in X}\{f(x,u)\mid x_0=a\}.
\]
We will proceed in two steps by first giving conditions for closedness of $\varphi_0$. The function $\varphi_0$ describes the dependence of the optimal value on $u\in M$ as well as on the initial state much like the cost-to-go function in the Hamilton-Jacobi theory of optimal control; see e.g.\ \cite{rw00b,gr02}. Once the closedness of $\varphi_0$ has been established, we can apply the classical recession criterion from finite-dimensional convex analysis to verify the closedness of $\varphi$.

The bilinear form
\[
\langle (a,u),(\tilde a,y)\rangle=a\cdot \tilde a+\langle u,y\rangle
\]
puts the space $\R^d\times M$ in separating duality with $\R^d\times C$. The following result establishes the lower semicontinuity of $\varphi_0$ with respect to the corresponding weak topology. The proof relies on regularity properties of differential equations much like the proof of \cite[Theorem 3']{roc76b}. We use the same interiority condition but we relax the continuity assumptions on the domain of the Hamiltonian.

\begin{theorem}\label{thm:ctg}
In addition to the hypotheses of Theorem~\ref{thm:pdo}, assume that there exists $\bar y\in\dom g\cap AC$ with $\bar y_t\in\interior\dom_2 H_t$ for all $t$. Then $\varphi_0$ is closed, the infimum in the definition of $\varphi_0$ is attained for every $(a,u)$ and
\[
\varphi_0^\infty(a,u)=\inf_{x\in X}\{f^\infty(x,u)\mid x_0=a\}.
\]
\end{theorem}

\begin{proof}
Note that $\varphi_0$ is the value function associated with $f_0:X\times(\R^d\times M)\rightarrow\ereals$ defined by
\begin{align*}
  f_0(x,(a,u)) &=f(x,u)+\delta_0(x_0-a).
\end{align*}
By Theorem~\ref{thm:dvf} below, it suffices to show that
\begin{equation}\label{eq:dvf}
  v\mapsto \inf_{(\tilde a,y)}\{f^*_0 (v,(\tilde a,y))-\langle (a,u),(\tilde a,y)\rangle\}
\end{equation}
is bounded above in a neighborhood of the origin for all $(a,u)$. 

The Lagrangian $L_0$ associated with $f_0$ can be written as
\[
L_0(x,(\tilde a,y))=L(x,y)-x_0\cdot \tilde a,
\]
so $f_0^*(v,(\tilde a,y))=f^*((v_{-1}+\tilde a,v|_{[0,T]}),y)$. By Theorem~\ref{thm:pdo},
\[
f_0^*(v,(\tilde a,y))=J_{\tilde K}(y,D(y-v|_{[0,T]}))+\tilde k(v_{-1}+\tilde a + y_0-v_0, y_T-v_T).
\]
It suffices to establish the existence of a continuous function $v\mapsto (\tilde a^v,y^v)$ from $V$ to $\R^d\times C$ such that $y^v\in AC$, $y^0=\bar y$, $y^v_T=\bar y_T$, $y^v_0+\tilde a^v=\bar y_0$, and such that the function
\begin{equation}\label{eq:bound}
v\mapsto \int  {\tilde K}_t(y^v+v_t,\dot y^v_t)d\mu_t
\end{equation}
is bounded above in a neighborhood of the origin. Indeed, we will then have
\begin{align*}
  &\inf_{(\tilde a,y)}\{f_0^*(v,(\tilde a,y))-\langle (a,u),(\tilde a,y)\rangle\}\\
  &=\inf_{\tilde a,y\in C\cap X}\{J_{\tilde K}(y+v,Dy)+\tilde k(y_0+\tilde a,y_T)-\langle (a,u),(\tilde a-v_{-1},y+v|_{[0,T]})\rangle\}\\
  &\le J_{\tilde K}(y^v+v,Dy^v)+\tilde k(y^v_0+\tilde a^v,y^v_T)-\langle (a,u),(\tilde a^v-v_{-1},y^v+v|_{[0,T]})\rangle\}
\end{align*}
so that \eqref{eq:dvf} is bounded from above on a neighborhood of the origin.

By \cite[Lemma~2]{roc71}, there is an $\bar r>0$ such that $\mathbb B(\bar y_t,\bar r)\subset\dom_2 H_t$ for all $t$; see \cite[p.~460]{roc71}. We can then choose $v^i\in \R^d$, $i=0,\dots,d$ and an $r>0$ such that $|v^i|<\bar r$ and $\mathbb B(0,r)$ belongs to the interior of the convex hull of $\{v^i\mid i=0,\dots, d\}$. Having assumed the hypotheses of Theorem~\ref{thm:pdo}, conditions 2 and 4 of Theorem~\ref{thm:L} give the existence of functions $z^i\in L^1$ and nonnegative $\beta^i\in L^1$ such that
\[
H_t(x,\bar y_t+v^i)-x\cdot z^i_t\ge -\beta^i_t.
\]
Taking infimum over $x\in\R^d$ gives
\begin{equation}\label{eq:ctg1}
{\tilde K}_t(\bar y_t+v^i,z^i_t)d\mu_t\le\beta^i_t.
\end{equation}
Let $Z_t=[z_t^1-z_t^0\ \dots\ z_t^d-z_t^0]$ and $W=[v^1-v^0\ \dots \ v^d-v^0]$. Then $W$ is nonsingular and
\begin{equation}
	z^i_t=A_t (\bar y_t+v^i)+b_t,\quad i=0,\dots,d,\label{eq:A}
\end{equation}
where $A_t= Z_t W^{-1}$ and $b_t= z^0_t- Z_t W^{-1}(\bar y_t+v^0)$. Moreover, the integrability of $z^i$ and boundedness of $\bar y$ imply that $t\mapsto A_t$ and $t\mapsto b_t$ belong to $L^1$. By Lemma~\ref{app:ody} below, there is a $y^v\in AC$ such that 
\begin{equation*}
  dy^v_t=F_t(y^v_t+v_t)d\mu_t, \quad y^v_T=\bar y_T,
\end{equation*}
where
\[
F_t(y)=\begin{cases}
\dot{\bar y}_t&\text{if } y=\bar y_t\\
(1-\frac{|y-\bar y_t|}{r})\dot{\bar y}_t+\frac{|y-\bar y_t|}{r}[A_t (\bar y_t+r\frac{y-\bar y_t}{|y-\bar y_t|})+b_t]&\text{otherwise};
\end{cases}
\]
moreover, $v\mapsto (\tilde a^v,y^v)$, where $\tilde a^v=\bar y_0-y^v_0$, is a continuous transformation from $V$ to $\R^d\times C$.  We have that $y^0=\bar y$, $y^v_T=\bar y_T$ and $y^v_0+\tilde a^v=\bar y_0$ for all $v$. Next we establish that \eqref{eq:bound} is bounded above in a neighborhood of the origin which will finish the proof.

Since $y^0=\bar y$, there is a $\delta>0$ such that $\|y^v+v|_{[0,T]}-\bar y\|<r$ whenever $\|v\|<\delta$. Denoting $\alpha_t=\frac{|y^v_t+v_t-\bar y_t|}{r}$ and $w_t=\alpha_t^{-1}(y^v_t+v_t-\bar y_t)$, we have that $y^v_t+v_t=(1-\alpha_t)\bar y_t+\alpha_t(\bar y_t+w_t)$ and consequently, by the definition of $F$,
\begin{align*}
	{\tilde K}_t(y^v_t+v_t,\dot y^v_t) &={\tilde K}_t(y^v_t+v_t,F_t(y^v_t+v_t))\\
&\le (1-\alpha_t){\tilde K}_t(\bar y_t,\dot{\bar y}_t)+\alpha_t {\tilde K}_t(\bar y_t+w_t,A_t(\bar y_t+w_t)+b_t).
\end{align*}
The function $w$ can be expressed as $w_t=\sum_{i=0}^d \alpha^i_t v^i_t$, where $\alpha^i:[0,T]\rightarrow \R$ are measurable with $\sum_{i=0}^d \alpha^i_t=1$.  Since $|w_t|<r$ for all $t$, we have $0\leq\alpha^i_t \leq 1$ for all $t$, so, by \eqref{eq:A} and \eqref{eq:ctg1}, 
\begin{align*}
	{\tilde K}_t(\bar y_t+w_t,A_t(\bar y_t+w_t)+b_t)&\le \sum_{i=0}^d \alpha^i_t {\tilde K}_t(\bar y_t+v^i,z^i_t)\le\sum_{i=0}^d \alpha^i_t \beta^i_t.
\end{align*}
We define $\beta_t=\max\{0,{\tilde K}_t(\bar y_t,\dot{\bar y}_t)\}+\sum_{i=0}^d \alpha^i_t \beta^i_t$ so that $\beta\in L^1$ and
\begin{align*}
  \int  {\tilde K}_t(y^v_t+v_t,\dot y^v_t)d\mu_t\le\int \beta_t d\mu_t
\end{align*}
whenever  $\|v\|<\delta$.
\qquad\end{proof}

Combining Theorem~\ref{thm:ctg} with the classical recession condition gives sufficient conditions for the closedness of $\varphi$.

\begin{theorem}\label{thm:cc}
In addition to the hypotheses of Theorem~\ref{thm:ctg}, assume that
\[
\{x\in X\mid f^\infty(x,0)\le 0\}
\]
is a linear space. Then $\varphi$ is closed and the infimum in \eqref{P_u} is attained for every $u\in M$.
\end{theorem}

\begin{proof}
By Theorem~\ref{thm:ctg},
\begin{align*}
  \{a\in\R^d\mid \varphi_0^\infty(a,0)\le 0\}	&= \{a\in\R^d\mid \exists x\in X:\, f^\infty(x,0)\le 0,\, x_0=a\},
\end{align*}
which is linear when $\{x\in X\mid f^\infty(x,0)\le 0\}$ is linear. Since
\[
\varphi(u)=\inf_{a\in\R^d}\varphi_0(a,u),
\]
the claim follows from Theorem~\ref{thm:ctg} and Corollary~\ref{cor:dvf} below.
\qquad\end{proof}

The linearity condition in Theorem~\ref{thm:cc} is analogous to the condition 
\[
	\{y\in AC\mid \int K^\infty_t(y_t,\dot y_t)d\mu_t+k^\infty(y_0,y_T)\le 0\}
\]
in \cite[Theorem 3]{roc72a}. Indeed, the recession function $f^\infty$ can be expressed in terms of $K^\infty$ and $k^\infty$ as follows.

\begin{lemma}\label{lem:recf}
Assume that $f$ is proper and closed and that there exist $z\in L^1$, $y\in L^\infty$ and $\beta\in L^1$ such that
\[
	K_t(x,u)\ge x\cdot z_t+u\cdot y_t-\beta_t\quad\mu\text{-a.e.}
\]
Then
\[
	f^\infty(x,u)=J_{K^\infty}(x,Dx+u)+k^\infty(x_0,x_{T+}).
\]
\end{lemma}
\begin{proof}
We may assume without a loss of generality that $K_t(0,0)=0$. By monotone convergence theorem,
\begin{align*}
J_K^\infty(x,\theta)&=\lim_{\alpha\nearrow\infty}\frac{1}{\alpha}J_K(\alpha x,\alpha \theta)\\
 	&=\lim_{\alpha\nearrow\infty}\int  \frac{1}{\alpha}K_t(\alpha x_t,\alpha (d\theta^a/d\mu)_t)d\mu_t+\int  K^\infty_t(0,d\theta^s)/d|\theta^s)|)_t)d|\theta^s)|_t\\
 	&=\int  K^\infty_t(x_t,(d\theta^a/d\mu)_t)d\mu_t +\int  K^\infty_t(0,d\theta^s)/d|\theta^s)|)_t)d|\theta^s)|_t\\
&= J_{K^\infty}(x,\theta).
\end{align*}
The expression then follows from the general fact that if $f_1,f_2$ are closed convex functions and $A$ is a continuous linear mapping such that $f_1\circ A+f_2$ is proper, then $(f_1\circ A+f_2)^\infty=f_1^\infty\circ A+f_2^\infty$.\qquad\end{proof}

The assumptions in Lemma~\ref{lem:recf} are satisfied under the assumptions of Theorem~\ref{thm:L} whenever $\dom_1 I_H\neq\emptyset$. Indeed, then we have that $\dom_2 I_H\neq\emptyset$ (see the proof of Theorem~\ref{thm:L}), so, by the definition of Hamiltonian, there exist $y\in C$, $z\in L^1$ and $\beta\in L^1$ such that
 \begin{equation*}
 	K_t(x,u)\ge x\cdot z_t+u\cdot y_t-\beta_t.
 \end{equation*}

Combining the previous results with the biconjugate theorem gives a dual representation for the value function. 
\begin{theorem}\label{thm:dr}
Assume that
\begin{enumerate}
\item $t\mapsto\dom_1 H_t$ is left-isc and left outer $\mu$-regular
\item $\emptyset\neq\{x\in X\mid \exists r>0:\, \mathbb B(x_t,r)\subset\dom_1 H_t\ \forall t\}\subset\dom_1 I_H$,
\item $t\mapsto\dom_2 H_t$ is isc and outer $\mu$-regular
\item $\{y\in C\mid y_t\in\interior\dom_2 H_t\,\forall t\}\subset \dom_2 I_H$,
\item there exists a $\bar y\in\dom g\cap AC$ with $\bar y_t\in\interior\dom_2 H_t$ for all $t$,
\item $\{x\mid J_{K^\infty}(x,Dx+u)+k^\infty(x_0,x_{T+})\le 0\}$ is a linear space.
\end{enumerate}
Then the infimum in \eqref{P_u} is attained for every $u$ and
\[
\varphi(u) = \sup_{y\in C\cap X}\{\langle u,y\rangle - J_{\tilde K}(y,Dy) - \tilde k(y_0,y_T)\}.
\]
\end{theorem}
\begin{proof}
In view of Lemma~\ref{lem:recf} and Theorem~\ref{thm:cc}, it suffices to show that conditions 3 and 4 in Theorem~\ref{thm:L} are satisfied.

Assume that $x\in\dom_1 I_H$ and let $h_t(y)=-H_t(x_t,y)$ so that $h^*_t(u)=K_t(x_t,u)$. By \cite[Lemma~2]{roc71}, there is an $r>0$ such that $\mathbb B(\bar y_t,r)\subset\dom_2 H_t$ for all $t$; see \cite[p.~460]{roc71}. Therefore, by condition 2, $t\mapsto h_t(\bar y_t+y)$ is summable whenever $|y|<r$, so, by \cite[Proposition 3G]{roc76}, there is a $w\in L^1$ such that $I_{K}(x,w)<\infty$. This implies that
\[
	H_t(x_t,y)\le -y\cdot w_t + K_t(x_t,w_t)\ \mu\text{-a.e.}
\]
Thus condition 3 in Theorem~\ref{thm:L} holds. Condition 4 in Theorem~\ref{thm:L} is verified similarly.
\qquad\end{proof}

\section{Optimality conditions}\label{sec:dr}

This section derives optimality conditions for problem \eqref{P_u} when $u=0$. That is, we will be looking at the problem
\begin{equation}\tag{P}\label{P}
\minimize \ \int K_t(x_t,\dot x^a_t)d\mu_t+\int  K_t^\infty(0,\dot x^s_t)d|Dx^s|_t + k(x_0,x_{T+})\ \ \text{over $x\in X$},
\end{equation}
where $\dot x^a= d(Dx^a)/d\mu$ and $\dot x^s=d(Dx^s)/d|Dx^s|$. We associate with \eqref{P} the problem
\begin{equation}\tag{D}\label{D}
\minimize \ \int \tilde K_t(y_t,\dot y^a_t)d\mu_t+\int \tilde K_t^\infty(0,\dot y^s_t)d|Dy^s|_t + \tilde k(y_0,y_T)\ \ \text{over $y\in C\cap X$}.
\end{equation}

For a mapping $S_t:\R^d\rightrightarrows\R^d$ with $t\mapsto\gph S_t$ closed-valued and measurable, and for a function $z\in X$ of bounded variation, we write $Dz\in S(z)$ if
\begin{align*}
  \dot z^a_t &\in S_t(z_t)\quad\text{$\mu$-a.e.},\\
  \dot z^s_t &\in S^s_t(z_t)\quad |Dz^s|\text{-a.e.},
\end{align*}
where the mapping $S^s_t:\R^d\rightrightarrows\R^d$ is defined for each $t$ as the {\em graphical inner limit} (see \cite[Chapter 5]{rw98}) of the mappings $(\alpha S_t)(z):=\alpha S_t(z)$ as $\alpha\downto 0$. Here $t\mapsto\gph S^s_t$ is closed-valued and measurable; see \cite[Theorem 14.20]{rw98}. In particular, $\{t\mid \dot z^a_t \in S_t(z_t)\}$ and $\{t\mid \dot z^s_t \in S^s_t(z_t)\}$ are measurable sets (see \cite[Section 14.B]{rw98}), so $Dz\in S(z)$ is indeed well-defined. This definition is inspired by \cite[Section~14]{roc78} where the Hamiltonian conditions were extended from absolutely continuous trajectories to trajectories of bounded variation. Indeed, by \cite[Theorem 12.37]{rw98}, $S^s_t(z)$ coincides with the recession cone of $S_t(z)$ whenever $S_t$ is maximal monotone and $z\in\dom S_t$. 

We say that $x\in X$ and $y\in C\cap X$ satisfy the {\em generalized Hamiltonian equation} if 
\begin{equation*}
D(x,y)\in\Pi\tilde\partial H(x,y),
\end{equation*}
where $\Pi(v,u)=(u,v)$ and
\begin{align*}
\tilde\partial H_t(x,y) &= \partial_x H_t(x,y)\times\partial_y[-H_t](x,y).
\end{align*}
Since $\partial\tilde H_t$ is maximal monotone, $(\tilde\partial H)^s_t$ equals the normal cone mapping $N_{\cl\dom H_t}$ of $\cl\dom H_t$; see \cite[Example 12.27, Theorem 12.37]{rw98} and \cite[Theorem 37.4]{roc70a}. Moreover, $t\mapsto\gph\partial\tilde H_t$ is closed-valued and measurable \cite[Example 12.8 and Theorem 14.56]{rw98} and consequently the generalized Hamiltonian equation is well-defined. 

When $\dom H_t=\reals^d\times\reals^d$ (no state constraints), we have $N_{\cl\dom H_t}(x,y)=\{0\}$, so feasible trajectories are necessarily absolutely continuous and the generalized Hamiltonian equation reduces to that studied e.g.\ in \cite{roc70c}. When $\dom_1 H_t=\R^d$ (no state constraints in \eqref{P}), we recover the optimality conditions of \cite{roc72a} for optimal control problems; see \cite[Lemma 4]{roc72a}. 

As usual $x\in X$ and $y\in C\cap X$ are said to satisfy the {\em transversality condition} if
\begin{align*}
  (y_{0},-y_{T}) &\in\partial k( x_0,x_{T+}).
\end{align*}

\begin{theorem}\label{thm:nec}
Assume that $t\mapsto\dom_1 H_t$ is left-outer $\mu$-regular and that $t\mapsto\dom_2 H_t$ is outer $\mu$-regular. Then $\inf(P)\ge-\inf(D)$. For $\inf(P)=-\inf(D)$ to hold with attainment at feasible $x$ and $y$ respectively, it is necessary and sufficient that $x$ and $y$ satisfy the generalized Hamiltonian equation and the transversality condition.
\end{theorem}

\begin{proof}
We have $x_t\in\dom_1 H_t$ and $y_t\in\dom_2 H_t$ $\mu$-a.e.\ so, by \cite[Theorem~1]{per12}, $x_t\in\cl\dom_1 H_t$ and $y_t\in\cl\dom_2 H_t$ for all $t$. Consequently, $K_t^\infty(0,x)\ge x\cdot y_t$ and $\tilde K_t^\infty(0,v)\ge v\cdot x_t$ for all $t$. We get
\begin{align}
\begin{split}\label{eq:fen1}
	K_t(x_t,\dot x^a_t)+\tilde K_t(y_t,\dot y^a_t)&\ge x_t\cdot \dot y^a_t+ y_t\cdot\dot x^a_t\qquad\mu\text{-a.e.},\\
	K_t^\infty(0,\dot x^s_t)&\ge \dot x^s_t\cdot y_t\qquad |Dx^s|\text{-a.e.},\\
	\tilde K_t^\infty(0,\dot y^s_t)&\ge \dot y^s_t\cdot x_t\qquad |Dy^s|\text{-a.e.},\\
	k(x_0,x_{T+}) + \tilde k(y_{0},y_T) &\ge x_0\cdot y_0-x_{T+}\cdot y_T.
\end{split}
\end{align}
Integration by parts gives
\begin{align}
\begin{split}\label{eq:fen3}
& J_K(x,Dx) + k(x_{0},x_{T+}) + J_{\tilde K}(y,Dy) + \tilde k(y_{0},y_{T})\\
& \ge \int y_t dx_t+\int x_tdy_t+x_0\cdot y_{0}-y_{T}\cdot x_{T+}=0,
\end{split}
\end{align}
where the inequality holds as equality if and only if the inequalities in \eqref{eq:fen1} hold as equalities almost everywhere. In particular, we get $\inf(P)\ge-\inf(D)$.
 
Since $\partial\delta_{\cl\dom_1 H_t}(x)=N_{\cl\dom_1 H_t}(x)$ and since $\tilde K^\infty_t(0,\cdot)$ is the support function of $\cl\dom_1 H_t$, we have $\tilde K_t(0,v)=x\cdot v$ if and only if $v\in N_{\cl\dom_1 H_t}(x)$. Similarly $K_t(0,u)=u\cdot y$ if and only if $u\in N_{\cl\dom_2 H_t}(y)$. By \cite[Theorem 37.5]{roc70a}, we have $K_t(x,u)+\tilde K_t(y,v)=x\cdot v+u\cdot y$ if and only if $(v,u)\in\tilde\partial H_t(x,y)$. Therefore, \eqref{eq:fen3} holds as an equality if and only if the generalized Hamiltonian equation and the transversality condition hold. 
\qquad\end{proof}

The conditions of Theorem~\ref{thm:nec} generalize those in \cite[Theorem 2]{roc76b}. Indeed, outer semicontinuous mappings are both left-outer $\mu$-regular and outer $\mu$-regular; see $(3)$ in \cite{per12}. On the other hand, in \cite[Theorem 2]{roc76b} both trajectories are allowed to be discontinuous. 

Combining Theorem~\ref{thm:nec} with \cite[Theorem 1(b)]{roc71b} we obtain the following, where the problem \eqref{P_AC} is defined in the introduction.

\begin{theorem}\label{thm:eds}
Assume that $\mu$ is the Lebesque measure and that
\begin{enumerate}
\item $t\mapsto \dom_1 H_t$ is left outer $\mu$-regular,
\item $t\mapsto\dom_2 H_t$ is outer $\mu$-regular,
\item for all $x\in\R^d$ there exist $w\in L^1$ and $\alpha\in L^1$ such that
\[
	H_t(x,y)\le -y\cdot w_t+\alpha_t \quad\text{$\mu$-a.e.},
\]
\item  there exist $z\in L^1$, $y\in L^\infty$ and $\beta\in L^1$ such that
\[
	H_t(x,y_t)\ge x\cdot z_t-\beta_t\quad\mu\text{-a.e.}
\]
\item  $\{y\in AC\mid \int\tilde K^\infty_t(y_t,\dot y_t)d\mu_t+\tilde k^\infty(y_0,y_T)\le 0\}$ is a linear space.
\end{enumerate}
Then $\inf (P_{AC})=\inf (P)=-\inf (D)$, the optimal values are finite and the infimum in \eqref{D} is attained by some $y\in AC$. In particular, $x\in X$ attains the infimum in \eqref{P} if and only if it satisfies the generalized Hamiltonian equation and the transversality condition with some $y\in AC$.
\end{theorem}
\begin{proof}
Condition 3 implies that $\dom_1 H_t=\R^d$ $\mu$-a.e. which together with condition 1 gives that $\dom_1 H_t=\R^d$ for all $t$. Hence we have that $\tilde K_t^\infty(0,v)=\delta_{0}(v)$ and consequently $J_{\tilde K}(y,Dy)=+\infty$ unless $y\in AC$.

By condition 3 and by the definition of the Hamiltonian, for every $x\in\R^d$ there exist functions $w\in L^1$ and $\alpha\in L^1$ such that
\[
	K_t(x,w_t)\le \alpha_t\quad\text{$\mu$-a.e.}
\]
Similarly condition 4 implies that there exist functions $z\in L^1$, $y\in L^\infty$ and $\beta\in L^1$ such that
\[
	K_t(x,u)\ge x\cdot z_t+u\cdot y_t-\beta_t\quad\mu\text{-a.e.}
\]
Therefore, the conditions $(A)$--$(C)$ and $D_0$ in \cite{roc71b} hold. Consequently, we get from condition 4 and \cite[Theorem 3]{roc71b} that the assumptions of \cite[Theorem 1.(b)]{roc71b} are satisfied, so $\inf (P_{AC})=-\inf (D)$, these optimal values are finite and the infimum in \eqref{D} is attained by some $y$. Combining these facts with Theorem~\ref{thm:nec} gives the rest of the claims.
\qquad\end{proof}

Conditions 3--5 of Theorem~\ref{thm:eds} are just reformulations of the assumptions of \cite[Theorem 1(b)]{roc71b} so that they are readily comparable with the other assumptions made in this article.

\section{Appendix}

The first part of this appendix is concerned with the general conjugate duality framework of Rockafellar~ \cite{roc74}. Accordingly, $X$ and $U$ denote arbitrary locally convex topological vector spaces in separating duality with $V$ and $Y$, respectively. We fix a proper closed convex function $f:X\times U\rightarrow\ereals$ and denote the associated value function by
\[
\varphi(u)=\inf_{x\in X}f(x,u).
\]
Given $u\in U$, we define the extended real-valued function $\gamma_u$ on $V$ by
\[
\gamma_u(v)=\inf_{y\in Y}\{ f^*(v,y)-\langle u,y\rangle\}.
\]
Note that the domain of $\gamma_u$ equals
\[
\Gamma := \{v\mid \exists y: f^*(v,y)<\infty\}
\]
for every $u$.

The {\em recession function} $h^\infty$ of a closed proper convex $h:U\rightarrow\ereals$ is defined by 
\[
h^\infty(u)=\sup_{\alpha>0}\frac{h(\alpha u+\bar u)-h(\bar u)}{\alpha},
\]
where the supremum is independent of the choice of $\bar u\in\dom h$; see \cite[Theorem~8.5]{roc70a} for a proof in the finite-dimensional case. The recession function is sublinear and closed whenever $h$ is closed; see \cite{roc66}.

\begin{theorem}\label{thm:dvf}
Assume that, for every $u$, the function $\gamma_u$ is bounded from above on a neighborhood of the origin relative to $\aff\Gamma$. Then $\varphi$ is closed and proper, the infimum in the definition of $\varphi$ is attained for every $u\in U$ and
\[
\varphi^\infty(u)=\inf_{x\in X}f^\infty(x,u).
\]
\end{theorem}

\begin{proof}
Assume first that $\aff\Gamma=V$. Since $\gamma_u$ is convex (see e.g.\ \cite[Theorem~1]{roc74}) and bounded from above on a neighborhood of the origin, we have that $\gamma_u^*=f(\cdot,u)$ is inf-compact and $\gamma_u^{**}(0)=\gamma_u(0)$ (see e.g. \cite[Theorem 10]{roc74}). Therefore
\[
\varphi(u)=\inf_{x\in X}f(x,u)=-\gamma_u^{**}(0)=-\gamma_u(0)=\sup_{y\in Y}\{\langle u,y\rangle-f^*(0,y)\},
\]
where the infimum is attained and the last expression is closed in $u$. This implies together with the properness of $f$ that $\varphi$ is closed and proper.

Let $\bar u\in\dom\varphi$ and $\bar x\in X$ be such that $\varphi(\bar u)=f(\bar x,\bar u)$. We have that
\begin{align*}
\varphi^\infty(u) &= \sup_{\alpha>0}\frac{\varphi(\bar u + \alpha u)-\varphi(\bar u)}{\alpha}\\
&= \sup_{\alpha>0}\inf_{x\in X}f_\alpha(x,u),
\end{align*}
where
\[
f_\alpha(x,u) = \frac{f(\bar x+\alpha x,\bar u + \alpha u)-f(\bar x,\bar u)}{\alpha}.
\]
Clearly,
\[
\varphi^\infty(u) \le \inf_{x\in X}\sup_{\alpha>0}f_\alpha(x,u) = \inf_{x\in X} f^\infty(x,u).
\]
To prove the converse, let $\beta>\sup_{\alpha>0}\inf_{x\in X}f_\alpha(x,u)$ and let
\[
B_\alpha=\{x\in X\mid f_\alpha(x,u)\le\beta\}.
\]
The functions $f_\alpha$ are non-decreasing in $\alpha$, so the sets $B_\alpha$ are non-increasing in $\alpha$. Since the functions $x\mapsto f_\alpha(x,u)$ inherit inf-compactness from $x\mapsto f(x,u)$, we get, by the finite intersection property, that there is an $x'\in X$ with $x'\in B_\alpha$ for every $\alpha>0$. Thus,
\[
\sup_{\alpha>0}f_\alpha(x',u)\le\beta
\]
or in other words, $f^\infty(x',u)\le\beta$. Since $\beta>\sup_{\alpha>0}\inf_{x\in X}f_\alpha(x,u)$ was arbitrary, we have
\[
\inf_{x\in X} f^\infty(x,u) \le \sup_{\alpha>0}\inf_{x\in X}f_\alpha(x,u) = \varphi^\infty(u),
\]
which completes the proof for the case $\aff\Gamma=V$.

We now turn to the general case $\tilde V:=\aff\Gamma\subseteq V$. Let $N=\{x\in X\mid \langle x,\tilde v\rangle=0\ \forall \tilde v\in \tilde V\}$, $[x]=x+N$ and $X/N=\{[x]\mid x\in X\}$. By Hahn-Banach theorem, every continuous linear functional on $\tilde V$ extends to an element of $X$. On the other hand, $x'\in X$ and $x\in X$ define the same continuous linear functional on $\tilde V$ if and only if $x'\in [x]$. Thus $X/N$ can be identified with the continuous dual of $\tilde V$ with the pairing $\langle [x],\tilde v\rangle=\langle x,\tilde v\rangle$. 

Defining $\tilde f:X/N\times U\to\ereals$ by $\tilde f=(f^*|_{\tilde V\times Y})^*$, we have
\begin{align*}
  \tilde f([x],u) =\sup\{\langle x,\tilde v\rangle+\langle u,y\rangle -f^*(\tilde v,y)\}=f(x,u)
\end{align*}
and 
\[
\varphi(u)=\inf_{[x]\in X/N} \tilde f([x],u).
\]
Since $\tilde f^*=f^*|_{\tilde V\times Y}$, we can apply the first part of the proof to the conjugate duality framework corresponding to $\tilde f$. Thus $\varphi$ is closed, the infimum in the definition of $\varphi$ is attained and since $\tilde f^\infty([x],u)=f^\infty (x,u)$, we get
\[
\varphi^\infty(u)=\inf_{[x]\in X/N}  \tilde f^\infty([x],u)= \inf_{[x]\in X/N}f^\infty(x,u)=\inf_{x\in X}f^\infty(x,u),
\]
which finishes the proof.
\qquad \end{proof}

The following corollary was used in the proof of Theorem~\ref{thm:cc}.
\begin{corollary}\label{cor:dvf}
Assume that $X=R^d$ and that
\[
	\{x\,|\,f^\infty(x,0)\le 0\}
\] 
is a linear space. Then $\varphi$ is closed and proper, the infimum in the definition of $\varphi$ is attained for every $u\in U$ and
\[
\varphi^\infty(u)=\inf_{x\in X}f^\infty(x,u).
\]
\end{corollary}
\begin{proof}
Since $\gamma_u$ is now a convex function on $\R^d$, it suffices to show that the origin belongs to the relative interior of $\dom\gamma_u=\dom\gamma_0$; see \cite[Theorem 10.1]{roc70a}. By \cite[Theorem~7.4.1]{roc70a}, we have $\rinterior\dom\gamma_0=\rinterior\dom\cl\gamma_0$ while, by \cite[Corollary 13.3.4(b)]{roc70a}, $0\in\rinterior\dom\cl\gamma_0$ if and only if
\[
\mathcal L=\{x\,|\, (\gamma_0^*)^\infty(x)\le 0\}
\]
is a linear space. By \cite[Theorem~8.7]{roc70a},
\[
\mathcal L=\{x\,|\,\gamma_0^*(x)\le 0\}^\infty = \{x\,|\, f(x,0)\le 0\}^\infty=\{x\,|\, f^\infty(x,0)\le 0\},
\]
where we have used the fact that $\gamma_u^*(x)=f(x,u)$, by definition.
\qquad\end{proof}

The following lemma was used in the proof of Theorem~\ref{thm:ctg}. Its proof is rather standard in the case when $\mu$ is the Lebesque measure.

\begin{lemma}\label{app:ody}
Let $F:\R^d\times[0,T]\rightarrow\R^d$ be jointly measurable. Assume that there exists a $c\in L^1$ such that $|F_t(y^1)-F_t(y^2)|\leq |y^1-y^2|c_t$ and $|F_t(y^1)|\leq (1+|y^1|)c_t$ for all $t$ and $y^1,y^2\in\R^d$. Then for every $a\in\R^d$ and $v\in C$ there exists a unique $y^v\in AC$ such that
\begin{equation}
		dy^v_t=F_t(y^v_t+v_t)d\mu_t,\quad y_0=a.\label{eq:ody}
\end{equation}
Moreover, the mapping $v\mapsto y^v$ is continuous.
\end{lemma}
\begin{proof} 
Define $\mathcal T_v:C\rightarrow C$ by
\begin{align*}
	(\mathcal T_v y)_t=a+\int_{[0,t]}F_s(y_s+v_s)d\mu_s.
\end{align*}
Let $\gamma_t=\int_{[0,t]}c_sd\mu_s$. For any $y^1,y^2\in C$ we have
\begin{align*}
	|(\mathcal T_v y^1)_t-(\mathcal T_v y^2)_t|\leq \int_{[0,t]}|y^1_s-y^2_s|d\gamma_s,
\end{align*}
so, by induction, $|(\mathcal T_v^\nu y^1)_t-(\mathcal T_v^\nu y^2)_t|\leq \|y^1-y^2\|\frac{(\gamma_t)^\nu}{\nu!}$. For $\nu$ large enough, $\mathcal T_v^\nu$ is a contraction and $\mathcal T_v$ has a unique fixed point, i.e., there is a unique $y^v\in C$ satisfying \eqref{eq:ody}. 

Let $r>0$. For every $v\in C$ with $\|v\|<r$, we have 
\begin{align*}
	|y^v_t| &\leq |a|+\int_{[0,t]} (1+r+|y^v_s|)d\gamma_s\\
	&\le|a|+(1+r)\gamma_T+\int_{[0,t]}|y^v_s|d\gamma_s,
\end{align*}
so, by Gronwall's inequality (\cite[p. 498]{ek86}), $\|y^v\|\leq (|a|+(1+r)\gamma_T)e^{\gamma_T}$. Therefore, for every $v\in C$ with $\|v\|<r$, there is a $\beta\in\R$ such that
\begin{align*}
	|y^v_t-y^v_{t'}|\leq \int_{[t,t']}(1+|y^v_s|)d\gamma_s\leq\beta (\gamma_{t'}-\gamma_t).
\end{align*}
Thus the set $\{y^v\mid \|v\|<r\}$ is uniformly bounded and equicontinuous. Assume that $v\mapsto y^v$ is not continuous. Then there is a sequence $(v^{\nu})_{\nu=1}^\infty$  converging to $v$ such that $y^{v^\nu}\rightarrow\hat y$ and $\hat y\neq y^v$.  By dominated convergence and \eqref{eq:ody},
\[
	\hat y_t=a+\int_{[0,t]}F_s(\hat y_s+v_s)d\mu_s\quad\forall t
\]
so that, by the uniqueness of the fixed point of $\mathcal T_v$, we get $\hat y=y^v$, which is a contradiction.
\qquad\end{proof}

\bibliographystyle{siam}
\bibliography{sp}

\end{document}